\makeatletter

\documentclass[12pt]%
{article}

\newsavebox{\savepar}

\usepackage{amsmath}
\usepackage{amsfonts}
\usepackage{euscript}
\usepackage{oldgerm}
\usepackage{amsthm}
\usepackage{rotating}
\usepackage{mathrsfs}
\usepackage{hyperref}
\usepackage{amssymb}
\usepackage{epsfig}
\usepackage{dsfont}
\usepackage{subfig}
\makeindex

  \newtheorem{theorem}{Theorem}[section]

  \theoremstyle{definition}
\newtheorem{definition}[theorem]{Definition}

  \theoremstyle{remark}
  \newtheorem{remark}[theorem]{Remark}

\newtheorem{lemma}[theorem]{Lemma}

\theoremstyle{definition}

\theoremstyle{remark}
\begin{document}

%\frontmatter
\newcommand{\norm}[1]{\left\lVert #1\right\rVert}
\newcommand{\namelistlabel}[1]{\mbox{#1}\hfil}
\newenvironment{namelist}[1]{%
\begin{list}{}
{
\let\makelabel\namelistlabel
\settowidth{\labelwidth}{#1}
\setlength{\leftmargin}{1.1\labelwidth}
}
}{%
\end{list}}

\newcommand{\inp}[2]{\langle {#1} ,\,{#2} \rangle}
\newcommand{\vspan}[1]{{{\rm\,span}\{ #1 \}}}
\newcommand{\R} {{\mathbb{R}}}

\newcommand{\B} {{\mathbb{B}}}
\newcommand{\C} {{\mathbb{C}}}
\newcommand{\N} {{\mathbb{N}}}
\newcommand{\Q} {{\mathbb{Q}}}
\newcommand{\LL} {{\mathbb{L}}}
\newcommand{\Z} {{\mathbb{Z}}}

\newcommand{\BB} {{\mathcal{B}}}

\title{Zabreiko's lemma in 2-normed space and its applications.  }
\author{ Akshay S. RANE \footnote{Department of Mathematics, Institute of Chemical Technology, Nathalal Parekh Marg, Matunga, Mumbai 400 019, India, email :  as.rane@ictmumbai.edu.in,} 
\hspace {1mm}
}
\date{ }
\maketitle
%The work of this author was supported by UGC FRP, India.
\begin{abstract}
  	We prove the Zabreiko's lemma in 2-Banach spaces. As an application we shall prove a version of the closed graph theorem and open mapping theorem.
\end{abstract}



\newpage\noindent 
%In this section, we present definitions and notations used in the paper.
\section{Introduction}
n-normed spaces can be found in \cite{Chen}, \cite{Gun1}, \cite{Gun2}, \cite{Gun3}, and  \cite{Gun4}. Open mapping theorem and closed graph theorem are fundamental theorems in functional analysis. In a usual normed space, Zabreiko \cite {Zab}
proved  a lemma due to which the major theorems of functional analysis, like open mapping theorem and closed graph theorem followed as corollaries. In case of 2-normed spaces \cite{Son} and \cite{Raji} proved open mapping and closed Graph theorem in the usual way as in usual normed spaces. In this paper we would like to extend the Zabreiko's lemma in $2-$normed spaces and the open mapping theorem and a version of the closed graph theorem follows as a consequence.
\section{Methods}
 We shall now introduce definitions and results which shall be used in this paper.
\begin{definition}
Let $X$ be a vector space over $\mathbb{R}$ of dimension $>1.$ An 2-norm is a real valued function from $X \times X$ satisfying the following properties.
\begin{enumerate}
	\item $\|x_1,x_2\|=0$ if and only if $x_1,x_2$ are linearly dependent.
	\item $\|x_1,x_2\|=\|x_2,x_1\|$.
	\item $\|\alpha x_1,x_2\|=|\alpha| \|x_1,x_2\|$ for $\alpha \in \mathbb{R}$.
	\item $\|x_1+x_1^\prime ,x_2\|\leq \|x_1 ,x_2\|+\|x_1^\prime ,x_2\|$
\end{enumerate}
Then  $(X,\|.,.\|)$ is called as a 2-normed space. 
\end{definition} 

\begin{definition}
A real valued  function $p$ on $X \times X$ satisfying properties $2$, $3$ and $4$ is called 2-seminorm. 
\end{definition}\noindent
We incorporate various definition which can be found in \cite{Son}, \cite{Freese} and \cite{Lal}.
\begin{definition}
	If $p$ is a 2-seminorm, then for $ e\in X,$ $p(x,e)$ is a semi-norm on $X$.
\end{definition}
\begin{definition}
	A sequence $x_n$ in $X$  is said to be convergent to $a \in X,$ if $$ \lim_{n \rightarrow \infty} \|x_n -a,y\|=0,$$ for every $y \in X.$
\end{definition}

\begin{definition}
	A sequence $x_n$ in $X$  is said to be Cauchy if $$ \lim_{n, m\rightarrow \infty} \|x_n -x_m,y\|=0,$$ for every $y \in X.$
\end{definition}

\begin{definition}
A 2-normed space	$X$  is said to be an $2-$Banach space if every Cauchy sequence in $X$ converges.
\end{definition}

\begin{definition}
An open ball in a 2-normed space is	$B_{ e_2}(a,r) := \left \lbrace x \in X : \|x-a,e_2\| <r \right \rbrace $ and a closed ball is 
	$B_{e_2}[a,r] := \left \lbrace x \in X : \|x-a,e_2\| \leq r \right \rbrace $\\
\end{definition}

\begin{definition}
	A subset $U$ of $X$ is said to be open in $X,$ if there exist $r>0$ and $e_2 \in X$  such that $ B_{ e_2}(a,r) \subset X$.
\end{definition}

\begin{definition}
	A linear operator $T: X \rightarrow Y $ is said to be $e$ bounded if there exist $M_e >0$ such that $$\|Tx,Te\| \leq M_e \|x,e\|,$$ where $X$ and $Y$ are $2-$normed spaces. The operator $T$ is said to be bounded if $T$ is $e$ bounded for each $x \in X.$
\end{definition}\noindent 
\section{Results}
We shall now prove the following lemma which will be used to prove the Zabreiko's Lemma.

\begin{lemma}
	Let $p$ be a seminorm on a $2-$normed space $(X,\|.,.\|)$ and for $\alpha>0$ consider $$ V_\alpha= \left \lbrace x \in X: p(x,e) \leq \alpha, \text{for}\; e \in X \right \rbrace$$
	Suppose there is $a \in X$ $r>0$ and $e_2 \in X$,  such that $$B_{e_2}[a,r] \subset \overline{V_\alpha}.$$ Then for every $\delta >0$, 
	$$B_{e_2}[0,\delta r] \subset \overline{V_{\delta\alpha}}.$$
\end{lemma}
\begin{proof}
	We shall prove the result for $\delta=1$. Let $ x \in B_{e_2}[0, r].$ 
	Then $$\|(x+a)-a,e_2\|\leq r,\; \|(-x+a)-a,e_2\|\leq r. $$
	This implies that $ x+a, -x+a \in B_{e_2}[a,r] \subset \overline{V_\alpha}$. There is a sequence $u_m$ and $v_m$ in $  V_\alpha $ such that $ u_m  \rightarrow x+a $ and $v_m \rightarrow -x+a$ with $p(u_m,e) \leq \alpha,$ and $p(v_m,e) \leq \alpha.$
	Hence $$ \frac{u_m-v_m}{2} \rightarrow x$$ with $$p(\frac{u_m-v_m}{2},e) \leq \alpha.$$
	So $x \in  \overline{V_\alpha}$ and we are done. Now consider $\delta >0$ and $$\|x,e_2\|\leq \delta r$$ Let $y=\frac{x}{\delta}.$ Then $\|y,e_2\| \leq r.$ This implies $y \in  \overline{V_\alpha}$. So that there is a sequence $y_m$ in $V_\alpha$ so that $y_m \rightarrow y$ and $p(y_m,e) \leq \alpha.$ And we have $$\delta y_m \rightarrow \delta y,\; p(\delta y_m,e) \leq \delta \alpha.$$
	This implies $x \in \overline{V}_{\delta \alpha}.$
\end{proof}\noindent
We shall now prove the Zabreiko's theorem in $2-$normed spaces. The Baire Category theorem which we shall use to prove Zabreiko's Lemma for $2-$normed spaces is proved in \cite{Raji}.
\begin{theorem}
	Let $p$ be a countabely subadditive seminorm on a $2-$Banach space. Then there is $e_2 \in X$ and a constant $M>0$ such that $$ p(x,e) \leq M \|x,e_2\|,$$ for any $e \in X.$ In particular,
	$$ p(x,e_2) \leq M \|x,e_2\|.$$ 
\end{theorem}
\begin{proof}
	For $n \in \mathbb{N}$,
	let $$ V_n = \left \lbrace x \in X: p(x,e) \leq n , e \in X \right \rbrace$$
	Then $$ X= \bigcup_{n=1}^\infty V_n = \bigcup_{n=1}^\infty \overline{V}_n$$
	As a consequence $$\bigcap_{n=1}^\infty  \overline{V}_n^c = \phi.$$
	By the Baire Category theorem atleast one of $\overline{V}_n^c $ will not be dense in $X.$ As a result $$ \overline{\overline{V}_m^c} \subset X.$$ That is there is $a \in X $ such that $ a $ does not belong to $ \overline{\overline{V}_m^c}.$
	There exist $r>0$ and $e_2 \in X$ such that $$B_{e_2}[a,r] \cap
	\overline{V}_m^c= \phi$$
	This means $$ B_{e_2}[a,r] \subset \overline{V}_m$$
	Let $x \in X$. Let $ \epsilon >0. $ Define $$ \epsilon_0= \frac{\|x,e_2\|}{r}$$ 
	and $$ \epsilon= m\sum_{k=1}^\infty \frac{\epsilon}{ m 2^k}= m \sum_{k=1}^\infty \epsilon_k.$$
Since $\|x,e_2\|=r \epsilon_0 ,$ implies $ x \in B_{e_2}[0,\epsilon_0 r] $. From the above lemma
$$  B_{e_2}[0,\epsilon_0 r] \subset \overline{ V}_{\epsilon_0 m}$$
$x \in  \overline{ V}_{\epsilon_0 m}$, for $\epsilon_1r$ and any $\tilde{e} \in X$
$$ B_{\tilde{e}}[x, \epsilon_1 r] \bigcap  V_{\epsilon_0 m} \neq \phi. $$
There exist $x_1 \in  B_{\tilde{e}}[x, \epsilon_1 r] \bigcap V_{\epsilon_0 m}. $
This implies \begin{equation}\|x-x_1,\tilde{e}\|\leq \epsilon_1 r,\; \; p(x_1,e) \leq \epsilon_0 m,\end{equation} for any $ \tilde{e}, e \in X.$ In particular for $ \tilde{e}=e_2$
$$\|x-x_1,e_2\|\leq \epsilon_1 r.$$
Let $u_1=x-x_1$ so that $ x=x_1+u_1$ and $\|u_1,e_2\|\leq \epsilon_1 r $. So that $ u_1 \in B_{e_2}[0,\epsilon_1 r]$
From the above lemma once again
$$  B_{e_2}[0,\epsilon_1 r] \subset \overline{ V}_{\epsilon_1 m}$$
$u_1 \in  \overline{ V}_{\epsilon_1 m}$, for $\epsilon_2r$ and any $\tilde{e} \in X$
$$ B_{\tilde{e}}[u_1, \epsilon_2 r] \bigcap  V_{\epsilon_1 m} \neq \phi $$
There exist $x_2 \in  B_{\tilde{e}}[u_1, \epsilon_2 r] \bigcap V_{\epsilon_1 m} $
This implies \begin{equation}\|u_1-x_2,\tilde{e}\|\leq \epsilon_2 r,\; \; p(x_2,e) \leq \epsilon_1 m,\end{equation} for any $ \tilde{e}, e \in X.$ In particular for $ \tilde{e}=e_2$
$$\|u_1-x_2,e_2\|\leq \epsilon_2 r$$
Let $u_2=u_1-x_2 $ so that $x=x_1+u_1= x_1+x_2+u_2$ and $\|u_2,e_2\| \leq \epsilon_2 r.$ Now using previous lemma $$u_2 \in B_{e_2}[0,\epsilon_2 r] \subset \overline{V}_{\epsilon_2 m}$$ We get $x_3$ such that
  \begin{equation}\|u_2-x_3,\tilde{e}\|\leq \epsilon_3 r,\; \; p(x_3,e) \leq \epsilon_2 m,\end{equation} for any $ \tilde{e}, e \in X.$ 
  Continuing this way we get, for each $k$, $ x_k,u_k \in X$ such that 
   \begin{equation}\|u_{k-1}-x_k,\tilde{e}\|\leq \epsilon_k r,\; \; p(x_k,e) \leq \epsilon_{k-1} m,\end{equation} for any $ \tilde{e}, e \in X.$ As a result
   $$ x=x_1+u_1=x_1+x_2+u_2=x_1+x_2+\ldots+x_k+u_k$$ Thus for any $\tilde{e} \in X$ we have
   $$ \|x-\sum_{k=1}^n x_k,\tilde{e}\| =\|u_n, \tilde{e}\|\leq  \epsilon_m r= \frac{\epsilon r }{m 2^n}\rightarrow 0 $$ as $ n \rightarrow \infty.$ We see that $$ x=\sum_{k=1}^\infty x_k.$$ By the countable subadditivity of $p$ we have
   	$$ p(x,e) \leq \sum_{k=1}^\infty p(x_k,e) \leq m \sum_{k=1}^\infty \epsilon_{k-1} = m\epsilon_0  + \epsilon.$$
   The above statement is true for any $\epsilon.$	This implies for $e \in X$ we have 
   	$$ p(x,e)\leq \frac{m}{r} \|x,e_2\|.$$ 
   	In particular for $e=e_2$
   	$$  p(x,e_2)\leq \frac{m}{r} \|x,e_2\|.$$
\end{proof}
\begin{remark}
	Let $T:X\rightarrow Y$ be linear operator. Then for $x,e \in X$ $$p(x,e)=\|Tx,Te\|$$ is an example of a $2-$seminorm on $X.$
\end{remark}

\begin{definition}
	A linear operator $T: X \rightarrow Y $ is said to be  closed if for a sequence $x_n \rightarrow x $ in $X$, $Tx_n \rightarrow y$ in $Y$ then $y=Tx.$
\end{definition}\noindent 

We require the following result.
\begin{theorem}
Let $X$ be a $2-$Banach space. Then if $\displaystyle \sum_{n=1}^\infty \|x_n,e\| < \infty$ for each $ e \in X.$ Then $\displaystyle  \sum_{n=1}^\infty x_n < \infty.$
\end{theorem}

\begin{proof}
	Let $$ s_n =\sum_{k=1}^n x_k$$For any $e \in X$ and $n>m$, consider $$\|s_n-s_m,e\|= \|\sum_{k=m+1}^n x_k,e\|\leq \sum_{k=m+1}^n \|x_k,e\| $$
	Since $\displaystyle \sum_{k=m+1}^n \|x_k,e\| \rightarrow 0$ as $ n,m \rightarrow \infty$ 
	This implies $s_n$ is a Cauchy sequence in $2-$Banach space and $\lim s_n$ exists. As a result $\displaystyle \sum_{k=1}^\infty x_k$ converges.
\end{proof} \noindent
In the following results, we apply the Zabreiko's Lemma to prove a version of the closed graph and open mapping theorem.
\begin{theorem}
	Let $X$ and $Y$ be $2-$Banach spaces and $T: X \rightarrow Y$  be a closed linear operator. Then $T$ is e bounded.
\end{theorem}

\begin{proof}
	We show that for $e \in X$ $$p(x,e)=\|Tx,Te\|$$ is a countabely subadditive 2 norm on $X.$ Let $x=\displaystyle \sum_{n=1}^\infty x_n$ with $\displaystyle \sum_{n=1}^\infty\|T(x_n ), Te\|< \infty.$ Since $Y$ is a $2-$Banach space by the above result $\displaystyle  y=\sum_{n=1}^\infty Tx_n < \infty.$
	Define $s_n =\displaystyle  \sum_{k=1}^n x_k $. Note that $s_n \rightarrow x$ and $T(s_n) \rightarrow y.$ As $T$ is closed $y=T(x)=\displaystyle \sum_{n=1}^\infty Tx_n $. 
	$$p(x,e)= \|Tx,Te\|\leq \sum_{n=1}^\infty \|Tx_n,Te\|= \sum_{n=1}^\infty  p(x_n,e) $$ By the Zabreiko's lemma there is an $M.0$ and $e_2 \in X$ such that $$\|Tx,Te_2\|= p(x,e_2) \leq M \|x,e_2\|,$$ which proves the result.
\end{proof}

\begin{theorem}
	Let $X$ and $Y$ be $2-$Banach spaces. Let $T: X\rightarrow Y$ be continuous and surjective. Then $T$ is an open map.
\end{theorem}
\begin{proof}
	Define $$p(y,e^{\prime})= \inf \left \lbrace \|x,e\| : T(x)=y,\; T(e)=e^{\prime} \right \rbrace $$ Since $T$ is surjective and $\|.,.\|$ is non negative, $p(y,e^\prime)$ is well defined. it is easy to see that $p(kx,e)=|k| p(x,e).$ We shall show that $p$ is countabely subadditive seminorm on $Y$. Let $y=\sum_{k=1}^\infty y_k$ with $$\sum_{k=1}^\infty p(y_k,e^\prime )< \infty.$$
	Let $\epsilon >0$. For $\frac{\epsilon}{2^k} $ there is $x_k \in X $ such that $T(x_k)=y_k$ and $$ \|x_k, e\| < p(y_k,e^\prime) + \frac{\epsilon}{2^k}$$
	This implies $$ \sum_{k=1}^\infty \|x_k,e\| < \sum_{k=1}^\infty p(y_k,e^\prime)+ \epsilon$$
	As a result $$ \sum_{k=1}^\infty \|x_k,e\| < \infty $$ and as $X$ is a 2-Banach space $x=\sum_{k=1}^\infty x_k $ converges. $T$ closed $$ T(x)= \sum_{n=1}^\infty T(x_n)=\sum_{n=1}^\infty y_n=y $$
	As a result $$p(y,e^\prime)=\|y,e^\prime\| \leq \|x,e\|\leq \sum_{k=1}^\infty \|x_k,e\|< \sum_{k=1}^\infty p(y_k,e^\prime) + \epsilon $$
\noindent 	
 By Zabreiko's Lemma, there exist $e_2 \in Y$, $M >0$ such that $ p(y,\tilde{e}) < M \|y,e_2\|.$
  There exist $x \in X$ such that $T(x)=y,T(e)=e^\prime $ and $$ \|x,e\| \leq M \|y,e_2\|.$$
  Consider an open set $E$ in X and let $x_0 \in X$. There exits $ \delta >0$ and $ e \in X$ such that $B_e(x_0,\delta) \subset E.$ Let $y \in Y$, $ \|y-T(x_0), e_2\| < \frac{\delta}{M}.$ We shall show that $y \in T(E).$ Since $y -T(x_0) \in Y $ there exist $x \in X$ such that $T(x)=y-T(x_0)$ and $$ \|x,e\| \leq M \|y-T(x_0),e_2\|< \delta$$
  Thus $$\|x+x_0-x_0,e \| =\|x\| < \delta $$ This implies $x+x_0 \in B_e(x_0,\delta) \subset E$ So that $ y=T(x) +T(x_0) \in T(E) $
\end{proof}
\section{Conclusions}
We prove the Zabreiko's Lemma in 2-normed spaces, but for n-normed s]paces can be generalized in a similar manner. However using the Zabreiko's Lemma we were able to prove that a closed linear operator is e-bounded. It would be interest to show when the operator is bounded or continuous. Also question could be asked is whether analog of uniform boundedness principle could be obtained from the Zabreiko's lemma.

\section{List of Abbreviations}
\begin{enumerate}
	\item $B_{ e_2}(a,r) := \left \lbrace x \in X : \|x-a,e_2\| <r \right \rbrace. $ 
	\item $B_{e_2}[a,r] := \left \lbrace x \in X : \|x-a,e_2\| \leq r \right \rbrace. $
	\item$\overline{V}:$ closure of the set $V.$
	
\end{enumerate}
	
	\section{Declarations}
	Availability of data and materials: Not applicable \\
	Competing interests : Not Applicable\\
	Funding : Not Applicable\\
	Authors' contributions : Not Applicable.\\
	Acknowledgements : UGC FRP program, India\\
	Authors' information :\\
	UGC Assistant Professor,\\
	Department o Mathematics,\\
	Institute of Chemical technology, Mumbai.

%The author would like to thank UGC FRP, INDIA  for their support.

\end{document}